\documentclass{amsart}

\usepackage{amsfonts,amssymb}
\usepackage{amsmath}

\newtheorem{thm}{Theorem}

\newtheorem{assum}{Assumption}

\begin{document}

\title[A note on algebraic Riccati equations]{A note on algebraic Riccati equations
associated with reducible singular $M$-matrices}


\author{Di Lu}
\address{Department of Mathematics and Statistics, 
University of Regina, Regina, 
SK S4S 0A2, Canada}
\email{ludix203@uregina.ca}

\author{Chun-Hua Guo}
\address{Department of Mathematics and Statistics, 
University of Regina, Regina,
SK S4S 0A2, Canada} 
\email{chun-hua.guo@uregina.ca}

\thanks{This work was supported in part by a grant from
the Natural Sciences and Engineering Research Council of Canada.}



\begin{abstract}
We prove a conjecture about the minimal nonnegative solutions of algebraic Riccati equations associated with 
reducible singular $M$-matrices. The result enhances our understanding of the behaviour of doubling algorithms for finding the minimal  nonnegative solutions. 
\end{abstract}

\maketitle

\section{Introduction} 

For the algebraic Riccati equation
\begin{equation}\label{NARE}
XCX-XD-AX+B=0,
\end{equation}
where $A, B, C, D$ are real matrices of sizes $m\times m, m\times n,
n\times m, n\times n$, respectively, 
a systematic study was done in \cite{guo01} when 
\begin{align}\label{Mm}
K=\left [ \begin{array}{cc} D & -C\\ -B & A \end{array} \right ]
\end{align}
is a nonsingular $M$-matrix or an irreducible singular $M$-matrix. 
The study was recently extended in \cite{guo13} to reducible singular $M$-matrices under suitable assumptions. 

A real square matrix $A$ is called a
$Z$-matrix if all its off-diagonal elements are nonpositive, so any $Z$-matrix $A$
can be written as $sI-B$ with $B\ge 0$. A $Z$-matrix $A$ is called an $M$-matrix if
$s\ge\rho(B)$, where $\rho(\cdot)$ is the spectral radius; it is a singular $M$-matrix if
$s=\rho(B)$ and a nonsingular $M$-matrix if $s>\rho(B)$.

Some regularity assumption is needed to guarantee the existence of a solution of the equation \eqref{NARE} 
associated with the $M$-matrix $K$. 
An $M$-matrix $A$ is said to be regular if $Av\ge 0$ for some $v>0$. 

The following result  is proved in \cite{guo13}. 

\begin{thm}\label{thmn}
Suppose the matrix $K$ in \eqref{Mm} is  a regular $M$-matrix.  
Then $(\ref{NARE})$ has a minimal nonnegative solution $\Phi$ and 
$D-C\Phi$ is a regular $M$-matrix, and the dual equation 
\begin{equation}\label{nare2}
YBY-YA-DY+C=0,
\end{equation}
has a minimal nonnegative solution $\Psi$ and 
$A-B\Psi$ is a regular $M$-matrix. Moreover, 
$I_m-\Phi \Psi$ and $I_n-\Psi \Phi$ are both regular $M$-matrices. 
\end{thm}

Associated with the matrix $K$ in \eqref{Mm} is the matrix 
\begin{equation}\label{Hm}
H=\left [\begin{array}{cc}I_n & 0\\
0 & -I_m
\end{array} \right ] K=\left [ \begin{array}{cc}
D & -C\\
B & -A
\end{array} \right ]. 
\end{equation}

When the matrix $K$ in \eqref{Mm} is regular singular $M$-matrix, the following assumption has been  introduced in \cite{guo13}. 

\begin{assum}\label{a1}
The matrix $H$ in \eqref{Hm} has only one linearly independent eigenvector corresponding to the 
zero eigenvalue of multiplicity $r\ge 1$. 
\end{assum}

It is known \cite{guo01} that the assumption is satisfied with $r=2$ when $K$ is an irreducible singular $M$-matrix.

Under  Assumption \ref{a1}, there are nonnegative nonzero vectors 
$\left[\begin{array}{c}
u_1\\
u_2
\end{array}\right ]$ and $\left[\begin{array}{c}
v_1\\
v_2
\end{array}\right ]$, where $u_1, v_1\in {\mathbb R}^n$ and $u_2, v_2\in {\mathbb R}^m$, 
 such that 
$$
K\left[\begin{array}{c}
v_1\\
v_2
\end{array}\right ]=0, \quad [u_1^T\;u_2^T] K=0.  
$$
They are each unique up to a scalar multiple \cite{guo13}. 

The purpose of this note is to provide an affirmative answer to a conjecture in \cite{guo13}, regarding the matrices 
 $I_m-\Phi \Psi$ and $I_n-\Psi\Phi$.

\section{The result}

The following result was conjectured to be true in \cite{guo13}, and was proved under the restrictive assumption that 
at least one of $\Phi$ and $\Psi$ is positive. 

\begin{thm}\label{conj1}
Let $K$ be a regular singular $M$-matrix  with Assumption \ref{a1}. 
If $u_1^Tv_1\ne u_2^Tv_2$, then $I_m-\Phi \Psi$ and $I_n-\Psi\Phi$ are nonsingular $M$-matrices. 
\end{thm}

\begin{proof}
By Theorem \ref{thmn}, $I_m-\Phi \Psi$ and $I_n-\Psi \Phi$ are both $M$-matrices. So we just need to show they are nonsingular when $u_1^Tv_1\ne u_2^Tv_2$. Since $I_n-\Psi \Phi$ is nonsingular if and only if 
$I_m-\Phi \Psi$ is nonsingular, we only need to show $I_m-\Phi \Psi$ is nonsingular. 
In view of 
$$
\left [\begin{array}{cc}
I_n & 0\\
-\Phi & I_m
\end{array} \right ] \left [\begin{array}{cc}
I_n & \Psi\\
\Phi & I_m
\end{array} \right ]=\left [\begin{array}{cc}
I_n & \Psi\\
0 & I_m-\Phi\Psi
\end{array} \right ],
$$
we need to show that the matrix  
$$
\left [\begin{array}{cc}
I_n & \Psi\\
\Phi & I_m
\end{array} \right ]
$$
is nonsingular. Since $\Phi$ and $\Psi$ are solutions of \eqref{NARE} and \eqref{nare2}, respectively, it is easily verified that 
\cite{guo13}
\begin{equation}\label{wh}
\left [\begin{array}{cc}
D & -C\\
B & -A
\end{array} \right ] \left [\begin{array}{cc}
I_n & \Psi\\
\Phi & I_m
\end{array} \right ]=\left [\begin{array}{cc}
I_n & \Psi\\
\Phi & I_m
\end{array} \right ]\left [\begin{array}{cc}
R & 0\\
0 & -S
\end{array} \right ],
\end{equation}
where $R=D-C\Phi$ and $S=A-B\Psi$ are $M$-matrices. Therefore, the eigenvalues of $R$ and $S$ are all in the closed right half plane, with $0$ being the only possible eigenvalue on the imaginary axis. 
When  $u_1^Tv_1\ne u_2^Tv_2$, we know from \cite{guo13} that one of the matrices $R$ and $S$ is singular and the other is nonsingular. 
It follows that the matrices $R$ and $-S$ have no eigenvalues in common. 

Let 
$$
W={\rm Ker} \left [\begin{array}{cc}
I_n & \Psi\\
\Phi & I_m
\end{array} \right ]. 
$$
We need to show $W=\{0\}$. 

For any $x\in W$, 
post-multiplying \eqref{wh} by $x$ shows that $Tx\in W$, where 
$T$ is the linear transformation from $\mathbb C^{m+n}$ to $\mathbb C^{m+n}$, defined by 
$$
T\left [\begin{array}{c}
y_1\\
y_2
\end{array} \right ]=\left [\begin{array}{cc}
R & 0\\
0 & -S
\end{array} \right ]\left [\begin{array}{c}
y_1\\
y_2
\end{array} \right ],
$$
where $y_1\in \mathbb C^n$ and $y_2\in \mathbb C^m$. 
Thus $W$ is an invariant subspace of the linear transformation $T$. 
Suppose $W\ne \{0\}$. Then we have $0\ne w\in W$ such that $Tw=\lambda w$ for some $\lambda \in \mathbb C$. 
Write $w=\left [\begin{array}{c}
w_1\\
w_2
\end{array} \right ],
$
where $w_1\in \mathbb C^n$ and $w_2\in \mathbb C^m$. 
We then have $Rw_1=\lambda w_1$ and $(-S)w_2=\lambda w_2$. 
Since $R$ and $-S$ have no eigenvalues in common, one of $w_1$ and $w_2$ must be a zero vector. 
It then follows from 
$$
\left [\begin{array}{cc}
I_n & \Psi\\
\Phi & I_m
\end{array} \right ]\left [\begin{array}{c}
w_1\\
w_2
\end{array} \right ]=0
$$
that $w_1$ and $w_2$ are both zero vectors. The contradiction shows that $W= \{0\}$. 
\end{proof} 

It has been explained in \cite{guo13} that when $K$ is a regular singular $M$-matrix  with Assumption \ref{a1} and 
$u_1^Tv_1\ne u_2^Tv_2$, doubling algorithms (such as SDA in \cite{glx06} and ADDA in \cite{wwl12}) 
can be used to find the minimal nonnegative solutions $\Phi$ and $\Psi$ simultaneously. 
But before Theorem \ref{conj1} is proved in this note, there were a few subtle issues associated with the doubling algorithms. 
 We now have the following modification of \cite[Theorem 15]{guo13} about the ADDA, which uses two parameters $\alpha$ and $\beta$. The ADDA is reduced to the SDA when $\alpha=\beta$. 

\begin{thm}\label{thm99}
Let $K$ be a regular singular $M$-matrix with Assumption \ref{a1} and 
$u_1^Tv_1\ne u_2^Tv_2$. Assume that 
$\alpha\ge  \max a_{ii}>0$ and $\beta\ge \max d_{ii}>0$. Then the ADDA is well defined with 
$I-G_kH_k$ and $I-H_kG_k$ being nonsingular $M$-matrices for each $k\ge 0$. 
Moreover, $E_0\le 0, F_0\le 0$,  $E_k\ge 0, F_k\ge 0$, 
$0\le H_{k-1}\le H_k\le \Phi$, $0\le G_{k-1}\le G_k\le \Psi$ for all $k\ge 1$, 
and 
$$
\limsup_{k\to \infty}\sqrt[2^k]{\|H_k-\Phi\|}\le r(\alpha, \beta), 
\quad 
 \limsup_{k\to \infty}\sqrt[2^k]{\|G_k-\Psi\|}\le r(\alpha, \beta),  
$$
where 
$r(\alpha, \beta)=
\rho\left ((R+\alpha I)^{-1}(R-\beta I)\right )\cdot \rho\left ((S+\beta I)^{-1}(S-\alpha I)\right )<1$ 
with $R=D-C\Phi, S=A-B\Psi$. 
\end{thm}

Since we have now  proved that $I-\Phi \Psi$ and $I-\Psi\Phi$ are nonsingular $M$-matrices
we can use the approach in \cite{gim07} to prove that the ADDA is well defined with 
$I-G_kH_k$ and $I-H_kG_k$ being nonsingular $M$-matrices for each $k\ge 0$ even when 
$\alpha= \max a_{ii}$ and $\beta= \max d_{ii}$. 
That $r(\alpha, \beta)<1$ in Theorem \ref{thm99} is already known in  \cite{guo13}.  
Thus, $H_k$ converges to $\Phi$ quadratically, and $G_k$ converges to $\Psi$ quadratically. 

By \cite[Theorem 2.3]{wwl12}, the parameters $\alpha= \max a_{ii}$ and $\beta= \max d_{ii}$ minimize 
 $r(\alpha, \beta)$ among all parameters $\alpha\ge \max a_{ii}$ and $\beta\ge  \max d_{ii}$.  
Therefore, we should normally use the optimal values $\alpha= \max a_{ii}$ and $\beta= \max d_{ii}$ for the ADDA. 
Before Theorem \ref{conj1} is proved, we avoid using  $\alpha= \max a_{ii}$ and $\beta= \max d_{ii}$, to ensure that 
the ADDA is well defined. 

The matrices $(I-G_kH_k)^{-1}$ and $(I-H_kG_k)^{-1}$ appear in the ADDA. With the proof of Theorem \ref{conj1}, 
we now know that, in Theorem \ref{thm99},  the matrices $I-G_kH_k$ and $I-H_kG_k$ will not converge to singular matrices. This is of course favorable for the ADDA.

\end{document}